\documentclass{amsart}
\usepackage{graphicx} 
\usepackage{Style}

\DeclareMathOperator{\Stab}{Stab}

\title{Some remarks on L-equivalence for cubic fourfolds and hyper-Kähler manifolds}
\author{Simone Billi}
\address[Simone Billi]{Fakultät für Mathematik und Informatik, Universität des Saarlandes, Campus E2.4, 66123 Saarbrücken, Germany}
\email{billi@math.uni-sb.de}

\author{Lucas Li Bassi}
\address[Lucas Li Bassi]{}
\email{lucas.libassi@gmail.com }
\thanks{The authors are members of INdAM-GNSAGA}
\date{}

\begin{document}

\maketitle

\begin{abstract}
    We prove that if two very general cubic fourfolds are L-equivalent then they are isomorphic, and we observe that there exist special cubic fourfolds which are L-equivalent but not isomorphic. When the cubic fourfolds are very general in certain Hassett divisors, we prove that if they are L-equivalent then they are also Fourier--Mukai partners. 
    We also provide further examples in support of the fact that L-equivalent projective hyper-Kähler manifolds should be D-equivalent, as conjectured by Meinsma.  
\end{abstract}
\section{Introduction}

\subsection{L-equivalence}
The \textit{Grothendieck ring of varieties} \(K_0(\Var_k)\) over a field \(k\) is the free abelian group generated by isomorphism classes of varieties over \(k\), modulo the relations
\[[X]=[X\setminus Z]-[Z]\]
for every closed subvariety \(Z\subset X\).
The ring structure is given by the product 
\[[X]\cdot [Y]:=[X\times_k Y],\]
which is commutative with unit the class of a point \([\Spec(k)]\). 
The class of the affine line 
\[\bbL:=[\mathbb{A}^1_k]\in K_0(\Var_k)\]
is called the \textit{Lefschetz motive}. 
The class of a variety in the Grothendieck ring of varieties is a motivic invariant, it also carries information about the birational geometry of the variety and about its derived category.

It was proved by Larsen and Lunts \cite{LarsenLunts:motivic} that the quotient \(K_0(\Var_k)/\bbL\) is isomorphic to the free abelian group generated by the stably birational classes of smooth projective varieties. Two varieties \(X,Y\) are called \textit{D-equivalent} if there is an equivalence \(D^b(X)\cong D^b(Y)\) of their bounded derived categories of coherent sheaves. As a consequence of results of Kontsevich \cite{Kontsevich:motivic_integration} (see also \cite[4.1]{Veys_motivic_integration}), a conjecture of Kawamata \cite[Conjecture 1.2]{kawamata2002d} predicts that birational D-equivalent smooth projective varieties have equal classes in the completion (with respect to the dimension filtration) of the localization \(K_0(\Var_k)[\bbL^{-1}]\). This is in fact proved to be true by Kawamata when the smooth projective varieties are D-equivalent and of general type \cite[Theorem 1.4 (2)]{kawamata2002d}, in which case they are showed to be automatically birational. It is then interesting to study the localization \(K_0(\Var_k)\to K_0(\Var_k)[\bbL^{-1}]\), with the difficulty that Borisov proved \cite{Borisov:class_of_affine_line} that the class \(\bbL\) is a zero-divisor in \(K_0(\Var_{\bbC})\). This lead Kuznetsov and Shinder to the following definition.
\begin{definition}[{\cite{kuznetsov2018grothendieck}}]
    Two varieties \(X\) and \(Y\) over \(k\) are \textit{L-equivalent} if there exists a positive integer \(r\) such that \[\bbL^{r}\cdot ([X]-[Y])=0\in K_0(\Var_k).\]
\end{definition}
 It was conjectured by Kuznetsov--Shinder \cite[Conjecture 1.6]{kuznetsov2018grothendieck} that simply connected algebraic varieties that are D-equivalent should be also L-equivalent, independently this was also conjectured by Ito--Miura--Okawa--Ueda \cite[Problem 1.2]{ito2020derived} without the assumption of simply connectedness.
By Bondal--Orlov's Reconstruction Theorem \cite[Theorem 2.5]{BondalOrlov}, if two varieties with ample canonical class or anticanonical class are D-equivalent then they are isomorphic. In particular, they are clearly L-equivalent. On the other hand, when the canonical bundle is trivial, this implication is proved to be false in \cite{efimov2018some,ito2020derived,meinsma2025counterexamples} including the simply connected case. Indeed, there exist examples of abelian varieties and projective hyper-Kähler manifolds that are D-equivalent but not L-equivalent. A variant of this conjecture \cite[Problem 7.2]{ito2020derived} is still an open case and it was proved to be true for abelian varieties in \cite{caucci2023derived}. The opposite implication is not entirely clear in general, and Meinsma gives the following conjecture in the case of hyper-Kähler manifolds.
\begin{conjecture}[Meinsma, {\cite[Conjecture 1.5]{meinsma2025counterexamples}}]
   Projective hyper-Kähler manifolds that are L-equivalent must be also D-equivalent.
\end{conjecture}
The results of this paper seek evidence for this conjecture, and aim to formulate a similar conjecture for cubic fourfolds.
\subsection{Results}
We will always assume \(k=\bbC\). The results we present are mainly based on the Hodge-theoretical and lattice-theoretical implications of L-equivalence, which were developed by Efimov \cite{efimov2018some} and Meinsma \cite{meinsma2024equivalence}.

Our first contribution is to determine whether L-equivalent cubic fourfolds need to be isomorphic (which is the same as being D-equivalent).

\begin{theorem}[{\autoref{thm:main_very_general_cubic}}]\label{thm:intro_very_general_cubic}
    Let \(Y\subset \mathbb{P}^5\) be a very general cubic fourfold. If \(Y'\) is a cubic fourfold which is L-equivalent to \(Y\), then \(Y\cong Y'\).
\end{theorem}
We say that a cubic fourfold \(Y\) is \textit{special} if it contains a surface \(S\) which is not homologous to the square of a hyperplane class \(h^2=\eta_Y\in A(Y):= H^4(Y,\bbZ)\cap H^{2,2}(Y)\), i.e. the \textit{algebraic lattice} \(A(Y)\) has rank at least \(2\). Special cubic fourfolds belong to \textit{Hassett divisors} \(\mathcal{C}_d\subset \mathcal{C}\) in the moduli space of cubic fourfolds with \(d\equiv 0,2 \mod 6\) and \(d>6\), see \cite{hassett2000special}. We say that a cubic fourfold is \emph{very general} if it does not belong to any Hassett divisor, i.e. it is not special. We say that a cubic fourfold is \emph{very general in the Hassett divisor \(\mathcal{C}_d\)} if it belongs to the Hassett divisor \(\mathcal{C}_d\) but does not belong to any other divisor \(\mathcal{C}_{d'}\) with \(d'\not= d\), such a cubic fourfold is in particular special.

We observe that \autoref{thm:intro_very_general_cubic} might be not true when considering a cubic fourfold which is not very general. 
\begin{remark}[{\autoref{prop:counterexamples}}]\label{remark: counterexamples}
  There exist special cubic fourfolds \(Y\) and \(Y'\) which are L-equivalent, but not isomorphic (and in particular not D-equivalent).\footnote{Similar general results about L-equivalent cubic fourfolds are independently achieved by Meinsma and Moschetti in \cite{meinsma2026lequivalencefouriermukaipartnerscubic}, where further examples of this phenomenon are studied.}
\end{remark}
The derived category of a cubic fourfold $Y$ admits a semi-orthogonal decomposition \cite{kuznetsov2009derived}:
\[
D^b(Y)=\langle \mathcal{A}_Y, \mathcal{O}_Y, \mathcal{O}_Y(1), \mathcal{O}_Y(2) \rangle,
\]
where the subcategory
\[
\mathcal{A}_Y:=\langle \mathcal{O}_Y, \mathcal{O}_Y(1), \mathcal{O}_Y(2) \rangle^\perp \subset D^b(Y)
\]
is called the \emph{Kuznetsov component}. The Bondal--Orlov reconstruction theorem implies that every derived equivalence between cubic fourfolds gives an isomorphism of the cubic fourfolds. Being D-equivalent turns out then to be a very strong condition in this case, it is therefore natural to consider a weaker notion of equivalence. Following  \cite{huybrechts2017k3}, we say that two cubic fourfolds are \emph{Fourier--Mukai (FM) partners} if there is a Fourier--Mukai equivalence
\(
\mathcal{A}_Y \cong \mathcal{A}_{Y'}
\)
between the corresponding Kuznetsov components. Such an equivalence does not necessarily extend to an equivalence \(D(Y)\cong D(Y')\).

The examples of \autoref{remark: counterexamples} are cubic fourfolds $Y,Y' \in \mathcal{C}_{20}$ studied by Fan and Lai \cite{Fan_Lai:New_rational}. They are birational to each other and, moreover, they are also Fourier--Mukai partners. \footnote{A similar situation is also expected in the case of cubic fourfolds containing pairs of non-syzygetic cubic scrolls, as in \cite{brooke2024birational}.}

\begin{theorem}[{\autoref{thm:FM_equiv}}]\label{thm:main_FM_equiv}
    Consider a very general cubic fourfold \(Y\in \mathcal{C}_d\) with \(d\equiv 0,2 \mod 6\) and \(d>6\) not divisible by \(9\). If \(Y'\) is a cubic fourfold which is L-equivalent to \(Y\), then \(Y\) and \(Y'\) are FM partners.
\end{theorem}

We then wonder if this is always true and pose the following. 

\begin{conjecture}
\label{thm:main_special_cubic}
     If \(Y\) and \(Y'\) are L-equivalent cubic fourfolds then they are FM partners.
\end{conjecture}

We now turn to the case of hyper-Kähler manifolds and provide additional evidence, in certain cases, for Meinsma's conjecture \cite[Conjecture 1.5]{meinsma2025counterexamples}, which predicts that L-equivalence implies D-equivalence for projective hyper-Kähler manifolds. We also recall that birational hyper-Kähler manifolds are conjectured to be D-equivalent \cite{bondal1995semiorthogonal,kawamata2002d}, and this is known to hold for manifolds of $K3^{[n]}$ type \cite{maulik2025d}. If \(X\) is hyper-Kähler, we consider the \textit{transcendental lattice} \(T(X)\subset H^2(X,\bbZ)\), the smallest primitive sublattice such that \(H^{2,0}(X)\subset T(X)\otimes \bbC\), and the \textit{Néron--Severi lattice} \(\NS(X)\subset H^2(X,\bbZ)\) which can be seen as the orthogonal to \(T(X)\) with respect to the Beauville--Bogomolov--Fujiki form, and it consists of algebraic classes. 
\begin{theorem}[{\autoref{prop:LSV}, \autoref{prop:LEquivalenceKuznetsov}, \autoref{cor:Kuz_special_cubics}, \autoref{prop:OG10_moduli_spaces}, \autoref{prop:hillbert_scheme}}]
    Let \(X\) and \(X'\) be projective hyper-Kähler manifolds which are deformation equivalent and L-equivalent. Assume that \(\End(T(X))=\bbZ\), then:
    \begin{enumerate}
        \item If \(X\) and \(X'\) are birational to the (twisted) LSV manifolds associated to very general cubic fourfolds, then \(X\) and \(X'\) are birational.
        
        \item If $X$ and $X'$ are moduli spaces on the Kuznetsov component of very general cubic fourfolds $Y$ and $Y'$ with Mukai vector of the same square, then $Y\simeq Y'$ and $X$ is birational to $X'$. 
        In case \(X\) and \(X'\) are of K3\(^{[n]}\) type, they are also D-equivalent.
        \item If \(X\) is of OG10 type, \(\NS(X)=U\oplus A_2(-1)\) and \(\disc T(X)=1\), then \(X\) and \(X'\) are birational and D-equivalent.
        \item If \(X\) is birational to \(S^{[n]}\) for a K3 surface \(S\) with \(U\subseteq \NS(S)\) and \(\rk T(S)\not=4\), then in case \(\disc T(X)=\disc T(X')=m\) is coprime to \(2(n-1)\) we have that \(X\) and \(X'\) are D-equivalent. 
    \end{enumerate}
\end{theorem}
Here \(\disc T(X)\) denotes the discriminant of \(T(X)\) as a lattice and \(\End(T(X))\) denotes the algebra of Hodge endomorphisms. We remark that the hypothesis \(\End(T(X))=\bbZ\) is satisfied by a general element of a lattice-polarized family, see \autoref{rmk:simple_end_algebra}.

\subsection*{Outline}

We collect in \autoref{sec:Hodge lattices} preliminaries and known results about lattices, Hodge structures, and L-equivalence. In \autoref{sec:cubic fourfolds} we present our results about cubic fourfolds. The subsections of \autoref{sec:hyper-Kähler} are dedicated to various constructions of hyper-Kähler manifolds involving cubic fourfolds and K3 surfaces. 

    \subsection*{Acknowledgements}
We are thankful to Lisa Marquand for pointing out a mistake in the previous version and for useful comments, also to Dominique Mattei for bringing the work \cite{maulik2025d} to our attention. 

The discussions about cubic fourfolds and the collaboration started at the IX Japanese-European Symposium on Symplectic Varieties and Moduli Spaces. We would like to thank the organizers for the nice environment of collaboration offered during the conference.

Simone Billi would like to thank Matthias Schütt and the Riemann Center for Geometry and Physics of Hannover for hosting and supporting him during the writing of this article.

\section{Hodge lattices and L-equivalence}\label{sec:Hodge lattices}

\subsection{Hodge lattices}
A \textit{lattice} is a finitely generated free abelian group \(L\) with a non-degenerate symmetric bilinear form \(b\colon L\times L\to \bbZ\). For elements \(v,w\in L\) we will write \(v^2:=b(v,v)\) and \(v\cdot w:=b(v,w)\). A lattice \(L\) is called \textit{even} if for any \(v\in L\) we have \(v^2\in 2\bbZ\), it is called \textit{odd} otherwise. We denote by \(O(L)\) the group of \textit{isometries}, i.e. isomorphisms preserving the bilinear form. We consider the dual lattice \(L^\vee:=\Hom(L,\bbZ)\cong L\otimes_\bbZ\bbQ\), which naturally contains \(L\), and let \[D(L):=L^\vee/L\] be the \textit{discriminant group}. It is a torsion group endowed with a symmetric bilinear form, which takes values in \(\bbQ/2\bbZ\) when \(L\) is even. We moreover let 
\[\disc(L):=|D(L)|\in \bbZ\]
be the \textit{discriminant} of \(L\), which is also equal to the absolute value of the determinant of a Gram matrix for \(b\). The length of the discriminant group \(l(D(L))\) is the smaller number of generators for \(D(L)\), it does not exceed the rank of \(L\). A lattice \(L\) is called \textit{unimodular} if \(\disc (L)=1\), or equivalently \(L=L^\vee\). 
The \textit{genus} of a lattice is the data of its signature, parity and discriminant form. The genus is clearly an isometry invariant, but there are non-isometric lattices with the same genus.

For an integer \(n\in \bbZ\) we denote by \(L(n)\) the lattice obtained by \(L\) multiplying the bilinear form by \(n\). We can allow also rational numbers \(q\in \bbQ\) and consider \(L(q)\), whose bilinear form takes rational values a priori, and it is a lattice if its bilinear form takes integer values. 

An injective morphism of lattices \(L\hookrightarrow M\) is called an \textit{embedding}, in this case we can see \(L\subset M\) as a sublattice of \(M\) and denote by \(L^\perp\subset M\) its orthogonal complement. An embedding of lattices \(L\subset M\) is called \textit{primitive} if the quotient \(M/L\) is torsion-free, it is called of \textit{finite index} if the quotient \(M/L\) is finite and in this case \(M\) is called an \textit{overlattice} of \(L\).

Following \cite{Nikulin_integral}, a primitive embedding \(L\subset M\) determines a finite index embedding \(L\oplus L^\perp \subseteq M\) and the group
\[G_{L,M}=\frac{M}{L\oplus L^\perp}\]
is embedded isometrically in \(D(L)\) and anti-isometrically in \(D({L^\perp})\). The group \(G_{L,M}\) is called the \textit{gluing subgroup}. It follows that 
\begin{equation}\label{eq:order_glueing}
    |G_{L,M}|^2=\frac{\disc(L)\cdot \disc(L^\perp)}{\disc(M)}.
\end{equation}

\begin{lemma}\label{lem:lenght_discriminant}
    Let \(L\subseteq \Lambda\) be a primitive sublattice of a unimodular lattice \(\Lambda\). Let \(T\) be a lattice and assume there exist \(n\in \bbZ\setminus \{0\}\) such that \(T(n)\) admits a primitive embedding in \(L\). If \(2\rk T >\rk \Lambda \) then \(n=\pm 1\).
\end{lemma}
\begin{proof}
    Assume by contradiction that \(n\not =\pm 1\). Then \(D({T(n)})\) contains a subgroup isomorphic to \((\bbZ/n\bbZ)^{\rk T}\), generated by an integral basis of \(T(n)\) multiplied by \(\frac{1}{n}\). Since \(\Lambda\) is unimodular, from (\ref{eq:order_glueing}) we have that the primitive embedding \(T(n)\subset \Lambda\) induces an isomorphism \(D({T(n)})\cong D(N)\), where \(N=T(n)^\perp\subset \Lambda\). This means that \(\rk T\leq \rk N=\rk \Lambda- \rk T\) which implies \(2\rk T\leq \rk \Lambda\), a contradiction. 
\end{proof}

\begin{definition}
    \begin{enumerate}
        
        \item[1)] A polarized torsion-free integral Hodge structure is called a \textit{Hodge lattice}. 
        \item[2)] A Hodge structure \(H\) of weight \(2n\) is called of \textit{K3 type} if it is torsion free, \(\dim_\bbC H^{n+1,n-1}=1\) and \(H^{n+i,n-i}=0\) for \(i=2,\dots, n\). 
        \item[3)]
        A Hodge lattice of K3 type is called \textit{irreducible} if it admits no primitive proper Hodge sublattice of K3 type.
    \end{enumerate}
\end{definition}

The polarization of a torsion-free Hodge structure gives a lattice structure which is compatible with the Hodge structure. In particular, a polarized Hodge structure of K3 type has the structure of a Hodge lattice. The main example of irreducible Hodge lattice will be the transcendental part of a Hodge structure of K3 type.
For a Hodge lattice \(H\), we denote by \(\End(H)\) the algebra of Hodge endomorphisms, by \(\Aut(H)\) the group of Hodge automorphisms and by \(O_{Hdg}(H)\) the group of Hodge isometries.

The following very useful and technical result about Hodge lattices will be fundamental.

\begin{proposition}[{\cite[Proposition 1.1 and Lemma 3.1]{meinsma2024equivalence}}]\label{prop:meinsma}
    Let \(T\) and \(T'\) be irreducible Hodge lattices of K3 type. Assume that \(\End(T)=\bbZ\) and that there is an isomorphism of Hodge structures \(T\cong T'\). Then there exists \(q\in \bbQ\) such that \(T'(q)\) is integral and there is a Hodge isometry \(T\cong T'(q)\).
\end{proposition}

If \(H\) is a polarizable Hodge structure of weight \(2n\), we consider the \textit{algebraic part} \(A(H):= H^{n,n}\cap H\). Let \(T(H)\) be the minimal integral Hodge sub-structure of \(H\) such that \[\bigoplus_{i=1}^n H^{n+i,n-i}\subset T(H)_\bbC,\] we call it the \textit{transcendental part}. We also set \[G(H):=\frac{H}{A(H)\oplus T(H)}.\] After fixing a polarization, we can consider the orthogonal complement \(A(H)^\perp\), in this case \((A(H)^\perp)^{n+i,n-i}=H^{n+i,n-i}\) for any \(i=1,\dots, n\) so that \(A(H)^\perp=T(H)\) if \(H\) is torsion-free. If \(H\) is torsion free and polarized, then it has the structure of a Hodge lattice. In particular, \(A(H)\) and \(T(H)\) are primitive sublattices and we get a finite index Hodge embedding \(A(H)\oplus T(H)\subseteq H\), so that \(G(H)=G_{A(H),H}\) is the gluing subgroup. 

We will be mainly interested in the case of Hodge structures of K3 type.

\begin{lemma}\label{lem:extension_to_Hodge_isom1}
    Let \(H\) and \(H'\) be polarized Hodge structures of K3 type of weight \(2n\), assume that \(H\) and \(H'\) are isometric as lattices and that there is a Hodge isometry \(f\colon T(H)\cong T(H')\). Then there exists a Hodge isometry \(h\colon H\cong H'\) restricting to \(f\) if and only if there is an isometry \(g\colon A(H)\cong A(H')\) such that \(f\) and \(g\) induce the same isometry from \(G(H)\) to \(G(H')\).
\end{lemma}
\begin{proof}
    On the lattice-theoretical level, this is precisely \cite[Corollary 1.5.2]{Nikulin_integral}. We then observe that such an \(h\) is Hodge if and only \(f\) is. In fact, since \(f=h_{\mid T(H)}\) then \(f\) is Hodge from \(T(H)\) to \(T(H')\) if and only if \(h\) is, moreover we have \(g(A(H)_\bbC)=A(H')_\bbC\) and also \(H_\bbC=A(H)_\bbC\oplus T(H)_\bbC\) since \(H\) is polarized (of K3 type), similarly for \(H'\). 
\end{proof}

\begin{lemma}\label{lem:extension_to_Hodge_isom2}
Let \(H\) and \(H'\) be polarized Hodge structures of K3 type of weight \(2n\) such that \(H\) and \(H'\) are isometric as lattices and that \(\rk A(H) \geq l(D({A(H)}))+2\). Assume that there is a Hodge isometry \(f\colon T(H)\cong T(H')\) and that \(A(H)\) and \(A(H')\) are even indefinite and in the same genus. Then there is an isometry \(g\colon A(H)\cong A(H')\) and a Hodge isometry \(h\colon H\cong H'\) restricting to \(f\) and \(g\).
\end{lemma}
\begin{proof}
    By \autoref{lem:extension_to_Hodge_isom1} it suffices to find an isometry \(g\colon A(H)\cong A(H')\) inducing the same isometry \(G(H)\cong G(H')\) as \(f\). From the fact that \(A(H)\) and \(A(H')\) are even indefinite and in the same genus with \(\rk A(H)\geq l(D({A(H)}))+2\), we know by \cite[Theorem 1.14.2]{Nikulin_integral} that they are isometric and that the map \(O(A(H))\to O(D({A(H)}))\) is surjective. Hence, we can always find the wanted isometry \(g\).
\end{proof}
\subsection{Implications of L-equivalence}

As in \cite{efimov2018some}, we consider the additive category \(\HS_{\bbZ}\) of integral polarizable Hodge structures and for any \(n\) the subcategory \(\HS_{\bbZ,n}\) of integral polarizable Hodge structures of weight \(n\). 
We consider the abelian group \(K_0(\Hdg_{\bbZ})\) generated by elements of \(\Hdg_{\bbZ}\) and subject to the relations \[[H_1\oplus H_2]=[H_1]+[H_2]\] for any \(H_1,H_2\in \Hdg_\bbZ\).
There is a homomorphism of additive groups
\[\Hdg_{\bbZ}\colon K_0(\Var_\bbC)\to K_0(\HS_{\bbZ})\] 
where \(\Hdg_{\bbZ}([X])=[H^*(X,\bbZ)]\), that factors through the localization \(K_0(\Var_\bbC)[\bbL^{-1}]\), see \cite{efimov2018some}. A similar definition can be given for \(\Hdg_{\bbZ,n}\).

For the following, we readapt the definitions and results given in \cite[Section 4]{meinsma2024equivalence}.
Consider an integral polarizable Hodge structure \(H\in \HS_{\bbZ,2n}\) of weight \(2n\). 
There is an additive endofunctor \(\mathcal{T}\colon \Hdg_{\bbZ,2n}\to \Hdg_{\bbZ,2n} \) associating to a Hodge structure its transcendental part. In particular, if \([H]=[H']\) then \([T(H)]=[T(H')]\).
Moreover, we have a group homomorphism 
\[\mathcal{G}\colon K_0(\HS_{\bbZ,2n})\to \bbQ^*\] 
such that \(\mathcal{G}([H])=|G(H)|\).

Recall that we have 
    \[|G(H)|^2=\frac{\disc(A(H))\cdot \disc(T(H))}{\disc(H)}.\]

\begin{lemma}\label{lem:unimod_Hodge_lattices}
    Let \(H\) and \(H'\) be unimodular Hodge lattices of K3 type, and let \(T\) and \(T'\) be their transcendental parts. Assume \([H]=[H']\in \HS_{\bbZ}\), then \(\disc(T)=\disc(T')\).
\end{lemma}
\begin{proof}
     Since the Hodge lattice \(H\) is unimodular, we have \(G(H)\cong D(T)\) and hence \(\mathcal{G}([H])=|D(T)|=\disc T\). Now, \(\disc(T)= \mathcal{G}([H])=\mathcal{G}([H'])=\disc(T')\).
\end{proof}

\begin{lemma}\label{lem:anti_isom}
    Let \(T\) and \(T'\) be Hodge lattices of K3 type. Assume that \(\End(T)=\bbZ\) and that \([T]=[T']\in K_0(\HS_\bbZ)\). If \(\disc T=\disc T'\) then there is a Hodge isometry \(T\cong T'(\pm1)\).
\end{lemma}
\begin{proof}
    Since \(\End(T)=\bbZ\), by \cite[Theorem 2.3]{efimov2018some} we have that there exists an isomorphism of Hodge structures \(T\cong T'\).
    In this case there is a Hodge isometry \(T\cong T'(q)\) by \autoref{prop:meinsma}, implying 
    \[\disc T= q^{\rk T}\disc T'\]
    from which we deduce that \(q=\pm 1\) since \(\disc T=\disc T\).
\end{proof}

\begin{remark}\label{rmk:simple_end_algebra}
      As a consequence of {\cite[Lemma 9]{geemen2016conjecture}}, the condition \(\End(T)=\bbZ\) is satisfied if \(T\) is the transcendental part of \(H^4(Y,\bbZ)\), or \(H^2(X,\bbZ)\), where \(Y\) is a cubic fourfold, and \(X\) is a hyper-Kähler manifold, very general in the family polarized by a certain lattice. In particular, this holds if \(Y\) is a very general cubic fourfold or \(Y\) is very general in \(\mathcal{C}_d\) (see also \cite[Lemma 2.3]{Fan_Lai:New_rational}). 
\end{remark}

\begin{remark}\label{rmk:signatures}
    In case we have an isometry of lattices \(T\cong T'\) and the signature of \(T\) is of the form \((l,m)\) with \(l\not=m\) then \(T\not\cong T'(-1)\). In particular this happens if \(\rk T\not=4\) and the signature is \((2,\rk T-2)\).
\end{remark}

\section{Cubic fourfolds}\label{sec:cubic fourfolds}
Let \(Y \subset \bbP^5_{\bbC}\) be a smooth cubic fourfold. Then \(H^4(Y,\bbZ)\) carries a polarized Hodge structure of K3 type of weight \(4\), which is an odd unimodular Hodge lattice of signature \((21,2)\). Indeed, cubic fourfolds are the archetypal example of Fano varieties of K3 type (see \cite{fatighenti2022topicsfanovarietiesk3} for the formal definition and a survey on the subject). From the points of view of Hodge theory and derived categories, these manifolds share many properties with K3 surfaces and more in general with hyper-Kähler manifolds. In several cases this relation is geometrically realized with a modular construction on the Fano variety (or its derived category), the historical motivating example is the variety of lines lying on a cubic fourfold \cite{beauville1985variete}. 

We denote by \(\eta_Y\in A(Y)=H^4(Y,\bbZ)\cap H^{2,2}(Y)\) the square of a hyperplane class, and let \(H^4_{prim}(Y,\bbZ)=\eta_Y^\perp \) be the \textit{primitive lattice} of \(Y\). As a lattice, \(H^4_{prim}(Y,\bbZ)\) is an even lattice isometric to \(U^{\oplus 2}\oplus E_8^{\oplus 2}\oplus A_2\). Recall that for a very general cubic fourfold we have \(A(Y)=\langle\eta_Y\rangle\). We also recall that for a very general cubic fourfold \(Y\in \mathcal{C}_d\not=\emptyset\) the lattice \(A(Y)\) is a positive-definite with \(\rk A(Y)=2\) and \(\disc A(Y)=d\), see \cite{hassett2000special}.
 
\begin{lemma}\label{lem:disc_cubic}
Let \(Y\) and \(Y'\) be L-equivalent cubic fourfolds, then \(\disc T(Y)=\disc T(Y')\).
\end{lemma}
\begin{proof}
    This is \autoref{lem:unimod_Hodge_lattices} with the fact that \(H^4(Y,\bbZ)\) is unimodular for a cubic fourfold \(Y\).
\end{proof}

\begin{lemma}\label{lem:Hodge_isometry_cubic}
    Let \(Y\) and \(Y'\) be L-equivalent cubic fourfolds such that \(\End(T(Y))=\bbZ\) and \(\rk T(Y)\not =  4\), then there is a Hodge isometry \(T(Y)\cong T(Y')\).
\end{lemma}
\begin{proof}
Since \(Y\) and \(Y'\) are L-equivalent we have that \([H^4(Y,\bbZ)]=[H^4(Y',\bbZ)]\in \HS_{\bbZ,4}\) and hence \([T(Y)]=[T(Y')]\). 
By \autoref{lem:anti_isom} and \autoref{lem:disc_cubic} we know that there is a Hodge isometry \(T(Y)\cong T(Y')(\pm1)\).
    The signature of \(T(Y)\) is given by \((2,\rk T(Y)-2)\), and by \autoref{rmk:signatures} we have a Hodge isometry \(T(Y)\cong T(Y')\) since \(\rk T(Y)\not = 4\) .
\end{proof}

\begin{theorem}\label{thm:main_very_general_cubic}
   Let \(Y\subset \mathbb{P}^5\) be a very general cubic fourfold. If \(Y'\) is a cubic fourfold which is L-equivalent to \(Y\), then \(Y\cong Y'\). In particular, \(Y\) and \(Y'\) are D-equivalent.
\end{theorem}

\begin{proof}
    Since \(Y\) is very general, we have \(T(Y)=H^4_{prim}(Y,\bbZ)\) and in particular \(\rk T(Y)=22\). Moreover, by \autoref{rmk:simple_end_algebra} we have \(\End(T(Y))=\bbZ\). In this case \autoref{lem:Hodge_isometry_cubic} gives a Hodge isometry \(T(Y)\cong T(Y')\). In particular, we have a Hodge isometry \(H^4_{prim}(Y,\bbZ)\cong H^4_{prim}(Y',\bbZ)\) and, up to composing it with \(-\id\), by \autoref{lem:extension_to_Hodge_isom1} this extends to a Hodge isometry \(H^4(Y,\bbZ)\cong H^4(Y',\bbZ)\) mapping \(\eta_Y\) to \(\eta_{Y'}\). Hence, by the Torelli Theorem for cubic fourfolds \cite{voisin1986theoreme} we can conclude that \(Y\cong Y'\).
\end{proof}

One might wonder if a similar statement holds for special cubic fourfolds. This happens not to be the case, for this we recall a construction given by Fan--Lai \cite{Fan_Lai:New_rational}. 

A very general cubic fourfold \(Y\in \mathcal{C}_{20}\) contains a Veronese surface \(V\subset \bbP^5\) and the system of quadrics containing \(V\) determines a Cremona transformation 
\[F_V\colon \bbP^5\dashrightarrow \bbP^5.\]
The indeterminacy locus of the birational inverse \(F_V^{-1}\) is a Veronese surface \(V'\subset \bbP^5\) projective equivalent to \(V\). 
\begin{proposition}[{\cite[Theorem 1.1]{Fan_Lai:New_rational}}]\label{thm:Fan-Lai}
    Let \(Y\in \mathcal{C}_{20}\) be a very general cubic fourfold containing a Veronese surface \(V\) and let \(Y'\) be the proper image of \(Y\) via \(F_V\). Then \(Y'\in \mathcal{C}_{20}\) is a cubic fourfold containing \(V'\), and it is the unique cubic fourfold such that \(Y\) and \(Y'\) are FM partners but \(Y\not\cong Y'\).
\end{proposition}

The previous examples are then easily seen to be L-equivalent. 
\begin{proposition}\label{prop:counterexamples}
 There exist special cubic fourfolds \(Y\) and \(Y'\) which are L-equivalent, but not isomorphic (and in particular not D-equivalent). Moreover, \(Y\) is a FM partner of \(Y'\) and we can choose \(Y\) so that \(\End(T(Y))=\mathbb{Z}\).   
\end{proposition}
\begin{proof}
   Let \(Y\in \mathcal{C}_{20}\) be a very general cubic fourfold, containing a Veronese surface \(V,\) and \(Y'\) be the proper image of the Cremona transformation \(F_V\). The indeterminacy locus of \(F_V\) is a Veronese surface \(V\subset Y\) and the indeterminacy locus of the inverse \(F_V^{-1}\) is a projective equivalent Veronese surface \(V'\). Thus, we can write the following relation 
    \[[Y]-[Y']=([Y\setminus V]+[V])-([Y'\setminus V']+[V'])=0\in K_0(\Var_\bbC)\]
    since \([V]=[V']\) and \([Y\setminus V] = [Y'\setminus V']\) from the above. In particular, \(Y\) and \(Y'\) are L-equivalent and by \autoref{thm:Fan-Lai} they are FM partners and not isomorphic. Since \(Y\) is very general in \(\mathcal{C}_{20}\), we have \(\End(Y)=\bbZ\) by \autoref{rmk:simple_end_algebra}.
\end{proof}

For a cubic fourfold \(Y\), we consider the lattice \(\widetilde{H}(\mathcal{A}_Y,\bbZ):=K_{top}(\mathcal{A}_Y)\) which is isometric to \(U^{\oplus 4}\oplus E_8(-1)^{\oplus 2}\). 
We denote by \(A(\mathcal{A}_Y)\) and \(T(\mathcal{A}_Y)\) the algebraic and transcendental sublattices, respectively. Moreover, \(\widetilde{H}(\mathcal{A}_Y,\bbZ)\) is endowed with a Hodge structure such that there is a Hodge isometry \(T(Y)(-1)\cong T(\mathcal{A}_Y)\), see \cite{Addington_rationality}.

\begin{theorem}\label{thm:FM_equiv}
    Consider a very general cubic fourfold \(Y\in \mathcal{C}_d\), for \(d\) not divisible by \(9\). If \(Y'\) is a cubic fourfold which is L-equivalent to \(Y\), then \(Y\) and \(Y'\) are FM partners.
\end{theorem}
\begin{proof}
    By \cite[Theorem 1.5 iii)]{huybrechts2017k3}, this is equivalent to prove that there is a Hodge isometry \(\widetilde{H}(\mathcal{A}_Y,\bbZ)\cong\widetilde{H}(\mathcal{A}_{Y'},\bbZ) \).
    
     Firstly, we observe that by \cite[Proposition 2.8]{AT} we have a Hodge isometry \(T(Y)(-1)\cong T(\mathcal{A}_Y)\). From \cite[Lemma 2.4]{Fan_Lai:New_rational} we know that the isometry class of \(T(Y)\) is determined by \(d\), and the discriminant group \(D({T(Y)})\) is cyclic when \(d\) is not divided by \(9\). Hence, in our hypothesis \(D({T(\mathcal{A}_Y)})\) is cyclic and also \(D({A(\mathcal{A}_Y)})\) is so from the fact that \(D({T(\mathcal{A}_Y)})\) is anti-isometric to \(D({A(\mathcal{A}_Y)})\) because \(\widetilde{H}(\mathcal{A}_Y,\bbZ)\) is unimodular.

    Since \(Y\) is very general in \(\mathcal{C}_{d}\), we have that \(\End(T(Y))=\bbZ\) by \autoref{rmk:simple_end_algebra} and also \(\rk T(Y)=21\). In this case \autoref{lem:Hodge_isometry_cubic} gives a Hodge isometry \(T(Y)\cong T(Y')\), in particular this implies that \(D({A(\mathcal{A}_Y)})\) and \(D({A(\mathcal{A}_{Y'})})\) are isometric. From the fact that \(l(D({A(\mathcal{A}_Y)}))=1\) and \(A(\mathcal{A}_Y)\) is indefinite with \(\rk A(\mathcal{A}_Y)=3\), we can apply \autoref{lem:extension_to_Hodge_isom2} to get a Hodge isometry \(\widetilde{H}(\mathcal{A}_Y,\bbZ)\cong\widetilde{H}(\mathcal{A}_{Y'},\bbZ) \).
\end{proof}

\begin{remark}
   Note that, according to a conjecture of Huybrechts \cite[Conjecture 3.21]{huybrechts2023geometry}, if two cubic fourfolds are FM partners then they should be birational. 
A well-known counterexample to the converse is given by two general Pfaffian cubic fourfolds, i.e.\ two general cubic fourfolds in the divisor $\mathcal{C}_{14}$. In fact, any two Pfaffian cubic fourfolds are birational since they are both rational by \cite{beauville1985variete}, but they are in general not FM partners by \cite[Theorem 1.5(iii)]{huybrechts2017k3}. 
The same example also shows that birationality does not imply L-equivalence for cubic fourfolds, by virtue of \autoref{thm:FM_equiv}.
\end{remark}

\section{Hyper-Kähler manifolds}\label{sec:hyper-Kähler}
A hyper-Kähler manifold \(X\) will be a compact Kähler manifold which is simply connected and such that \(H^0(X,\Omega^2_X)=\sigma_X\cdot \bbC\) for an everywhere nondegenerate holomorphic 2-form \(\sigma_X\). Since L-equivalence is for algebraic varieties, we will consider projective hyper-Kähler manifolds. In this case \(H^2(X,\bbZ)\) is a torsion-free polarized Hodge structure of K3 type of weight \(2\), where the lattice structure is given by the Beauville--Bogomolov--Fujiki form of signature \((3,b_2(X)-3)\). The algebraic part coincides with the Néron--Severi lattice \(\NS(X)=H^{1,1}(X)\cap H^2(X,\bbZ)\) and it transcendental part is given by \(T(X)=\NS(X)^\perp\), in analogy to the case of K3 surfaces (see for example \cite[Lemma 3.1]{huybrechts2016lectures}). Also, since \(X\) is projective, the signature of \(\NS(X)\) is of type \((1,-)\) \cite[Proposition 26.13]{gross2012calabi}. One can refer for example to \cite{gross2012calabi,debarre2022hyper} for generalities about hyper-Kähler manifolds.

Manifolds of \textit{K3\(^{[n]}\)} type are hyper-Kähler manifolds that are deformation equivalent to the Hilbert scheme of \(n\) points on a K3 surface \cite{beauville1983varietes}. Manifolds of \textit{OG10 type} are hyper-Kähler manifolds that are deformation equivalent to the symplectic resolution of a moduli spaces of sheaves on a K3 surfaces as in \cite{OG10}.

\begin{lemma}\label{lem:disc_same_deformation_disc_p}
Let \(X\) and \(X'\) be L-equivalent projective hyper-Kähler manifolds of the same deformation type such that \(\disc H^2(X,\bbZ)=p\) is a prime, then \(\disc T(X)=c\cdot \disc T(X')\) with \(c\in \{1,p,\frac{1}{p}\}\).
\end{lemma}
\begin{proof}
    We have \(m=|G(X)|=|G(X')|\), moreover \(m|\disc T(X)\) and \(m|\disc \NS(X)\) and similarly for \(X'\). Observing that from (\ref{eq:order_glueing}) \[p m^2= \disc T(X) \cdot \disc \NS(X)\] we deduce that \(\disc T(X)\in \{m, pm\}\), similarly for \(X'\).
\end{proof}

\begin{lemma}\label{lem:big_transcendental}
    Let \(X\) and \(X'\) be L-equivalent projective hyper-Kähler manifolds of the same deformation type and assume that \(\disc H^2(X,\bbZ)=p\) a prime number, \(\End(T(X))=\bbZ\) and \(\rk T(X)\not= 4\). If there is a primitive embedding \(H^2(X,\bbZ)\subseteq \Lambda\) in a unimodular lattice \(\Lambda\) with \(2 \rk T(X)>\rk \Lambda\), then there is a Hodge isometry \(T(X)\cong T(X')\).
\end{lemma}
\begin{proof}
     Since \(\End(T(M))=\bbZ\), then by \autoref{lem:anti_isom} together with \autoref{lem:disc_same_deformation_disc_p} there is a Hodge isometry \(T(X)\cong T(X')(q)\) for \(q\in\{\pm 1, \pm p, \pm \frac{1}{p}\}\). Up to exchanging \(X\) and \(X'\), we can assume \(q\in \bbZ\). From the hypothesis that there is a primitive embedding \(H^2(X,\bbZ)\subseteq \Lambda\) in a unimodular lattice \(\Lambda\) and \(2 \rk T(X)>\rk \Lambda\), we can apply \autoref{lem:lenght_discriminant} to deduce that \(q=\pm 1\). Since \(X\) and \(X'\) are projective, then \(T(X)\) and \(T(X')\) have signature \((2,-)\) and by \autoref{rmk:signatures} with \(\rk T(X)\not=4 \) we get a Hodge isometry \(T(X)\cong T(X')\). 
\end{proof}

\begin{remark}\label{rmk:lattices_for_defos}
We recall the following:
\begin{enumerate}
    \item  For \(X\) of OG10 type, we have that \(H^2(X,\bbZ)\cong U^{\oplus 3}\oplus E_8(-1)^{\oplus 2}\oplus A_2(-1)\) as lattices. In particular, \(\disc H^2(X,\bbZ)=3\) and there is a primitive embedding in the unimodular lattice \(\Lambda=U^{\oplus 5}\oplus E_8(-1)^{\oplus 2}\) of rank \(26\). Hence, \autoref{lem:big_transcendental} can be applied for \(\rk T(X)>13\).
    \item  For \(X\) of K3\(^{[2]}\) type, we have that \(H^2(X,\bbZ)\cong U^{\oplus 3}\oplus E_8(-1)^{\oplus 2}\oplus [-2]\) as lattices. In particular, \(\disc H^2(X,\bbZ)=2\) and there is a primitive embedding in the unimodular lattice \(\Lambda=U^{\oplus 4}\oplus E_8(-1)^{\oplus 2}\) of rank \(24\). Hence, \autoref{lem:big_transcendental} can be applied for \(\rk T(X)>12\).
\end{enumerate}
   
\end{remark}

\subsection{Hyper-Kähler manifolds from a cubic fourfold}

\subsubsection{The Fano variety of lines}
We firstly consider the Fano variety of lines \(F(Y)\) associated to a cubic fourfold \(Y\), this is a projective hyper-Kähler manifold of K3\(^{[2]}\) type \cite{beauville1985variete}.

\begin{proposition}\label{lem:L-equiv_Fano_of_lines}
    Let \(Y\) and \(Y'\) be L-equivalent cubic fourfolds, then their Fano varieties of lines \(F(Y)\) and \(F(Y')\) are L-equivalent.
\end{proposition}
\begin{proof}
   We know by \cite[Lemma 2.1]{okawa2021example} that \(Y^{[2]}\) and \(Y'^{[2]}\) are L-equivalent. Moreover, by \cite[Theorem 5.1]{galkin2014fano} we have that 
   \[\bbL^2([F(Y)]-[F(Y')])=([Y^{[2]}]-[Y'^{[2]}])+[\bbP^4]([Y]-[Y'])\in K_0(\Var_\bbC).\]
   The first part of the statement follows multiplying by \(\bbL^m\) where \(m\) is big enough so that \([Y^{[2]}]-[Y'^{[2]}]\) and \([Y]-[Y']\) are annihilated, recalling that the Grothendieck ring of varieties is commutative.
\end{proof}

In particular, if \(Y\) and \(Y'\) are cubic fourfolds as in \autoref{thm:Fan-Lai} then by \autoref{prop:counterexamples} and \autoref{lem:L-equiv_Fano_of_lines} we get that \(F(Y)\) and \(F(Y')\) are L-equivalent. However, it was already known by \cite[Theorem 4.1]{brooke2024cubic} that \(F(Y)\cong F(Y')\).

 \subsubsection{The LSV manifolds}\label{sec: LSV manifolds}
We now treat the case of projective hyper-Kähler manifolds that are birational to the Laza--Saccà--Voisin (LSV) manifold \(J(Y)\) (or its twisted version \(J^t(Y)\)) for a cubic fourfold \(Y\). These are hyper-Kähler manifolds of OG10 type obtained as compactifications of the fibration (twisted) intermediate Jacobian of hyperplane section of cubic fourfolds, we refer to \cite{laza2017hyper,voisin2016hyper} for their construction.

\begin{proposition}\label{prop:LSV}
    Let \(X\) be a manifold of OG10 type birational to \(J(Y)\) (or \(J^t(Y)\)) for a very general cubic fourfold \(Y\). Let \(T'\) be a cubic fourfold and \(X'\) be another manifold of OG10 type birational to \(J(Y')\) (or \(J^t(Y')\)). If \(X\) is L-equivalent to \(X\), then \(Y\cong Y'\) and \(X\) is birational to \(X'\).
\end{proposition}

\begin{proof}

   We have that \(T(X)\) is Hodge anti-isometric to \(T(Y)\), see \cite{laza2017hyper,Sacca_birationality_twisted} and \cite[Proposition 5.1]{billi2024non}. Since $Y$ is very general we know by \autoref{rmk:simple_end_algebra} that \(\End(T(X))=\bbZ\). In this case we have \(\rk T(X)=22\) and by \autoref{rmk:lattices_for_defos} we can apply \autoref{lem:big_transcendental} to deduce that there is a Hodge isometry \(T(X)\cong T(X')\). In particular, we obtain a Hodge isometry \(T(Y)\cong T(Y')\) and by the Torelli Theorem for cubic fourfolds \cite{voisin1986theoreme} with the fact that \(Y\) and \(Y'\) are very general, we get \(Y\cong Y'\).
    In conclusion, we have that \(X\) and \(X'\) are birational to \(J(Y)\) (or \(J^t(Y)\)).
\end{proof}

We have that if \(Y\) is very general then \(\NS(J(Y))\cong U\) and \(\NS(J^t(Y))\cong U(3)\), moreover there are rational Hodge isometries \(T_\bbQ(J(Y))\cong T_\bbQ(J^t(Y))\cong T_\bbQ(Y)(-1)\). This implies that there exists \(q\in \bbQ\) such that there is a Hodge isometry \(T(J(Y))\cong T(J^t(Y))(q)\), since there is an isometry \(T(J(Y))\cong T(J^t(Y))\cong U^{\oplus 2}\oplus E_8(-1)^{\oplus 2}\oplus A_2(-1)\) we then have a Hodge isometry \(T(J(Y))\cong T(J^t(Y))\). 

We then wonder if \(J(Y)\) and \(J^t(Y)\) can be L-equivalent, and notice that by \cite[Corollary 3.10]{Sacca_birationality_twisted} they are not birational.

\subsubsection{Moduli spaces on the Kuznetsov component}
We finally treat the case of hyper-Kähler manifolds that are constructed from moduli spaces of Bridgeland (semi-)stable objects in the Kuznetsov component \(\mathcal{A}_Y\) of a cubic fourfold \(Y\).

The first stability condition on \(\mathcal{A}_Y\) was constructed in \cite{BLMS}. We distinguish the connected component $\Stab^{\dagger}(\mathcal{A}_Y)\subset\Stab(\mathcal{A}_Y)$ of the manifold of stability conditions containing such a condition, and for 
a Mukai vector $v\in \widetilde{H}(\mathcal{A}_Y, \bbZ)$ and a stability condition $\sigma\in\Stab^{\dagger}(\mathcal{A}_Y)$, denote by $M_v(\sigma)$ the moduli space of $\sigma$-semistable objects in $\mathcal{A}_Y$ with Mukai vector $v\in A(\mathcal{A}_Y)$. We also recall that there are always two algebraic classes \(\lambda_1,\lambda_2\in A(\mathcal{A}_Y)\) spanning a primitive sublattice isometric to \(A_2\), and when \(Y\) is very general we have \(A(\mathcal{A}_Y)\cong A_2\).
\begin{proposition}[{\cite{BLMNPS,LPZ_OG10}}]
    Let \(v\in \widetilde{H}(\mathcal{A}_Y, \bbZ)\) and \(\sigma \in \Stab^\dagger(\mathcal{A}_Y)\) be a $v$-generic stability condition.

    \begin{enumerate}
    \item If \(v\) is primitive, then \(M_v(\sigma)\) is hyper-Kähler of K3\(^{[n]}\) type with \(n=\frac{v^2+2}{2}\).
    \item If \(v=2v_0\) with \(v_0\) primitive and \(v_0^2=2\), then $M_v(\sigma)$ admits a symplectic resolution $\widetilde{M}_v(\sigma)$ via blow-up which is a hyperkähler manifold of OG10 type.
\end{enumerate}
\end{proposition}

    For this reason, we will call a vector $v\in \widetilde{H}(\mathcal{A}_Y, \bbZ)$ of \emph{O'Grady type} if $v=2v_0$ for a primitive vector $v_0$ with $v_0^2=2$.\newline\newline
    Most of the propositions proved in the previous sections rely on finding some connection between the L-equivalence and Hodge theoretic properties of cubic fourfolds. This is why the following lemma is particularly useful in order to deal with $M_v(\sigma)$ for some $\sigma$ in $\Stab^{\dagger}(\mathcal{A}_Y)$ and $v\in\widetilde{H}(Y,\bbZ)$.

\begin{lemma}\label{trascendente_Kuz}
    Let $\sigma$ be a $v$-generic stability condition in $\Stab^{\dagger}(\mathcal{A}_Y)$ for either a primitive or a O'Grady type Mukai vector $v\in\widetilde{H}(Y,\bbZ)$. Then $T(M_v(\sigma))\simeq T(Y) (-1)$.
\end{lemma}
\begin{proof}
    We know by \cite[Proposition 2.8]{GGO} and \cite[Theorem 1.6]{BLMNPS} that in both cases there exists a Hodge isometry $H^2(M_v(\sigma), \mathbb{Z})\simeq
    v^{\perp}\subset \widetilde{H}(\mathcal{A}_Y, \mathbb{Z})$. Moreover, by \cite[Proposition 2.8]{AT} there exists another Hodge isometry $\langle\lambda_1,\lambda_2\rangle^\perp=
    A_2^{\perp}\simeq H^4_{prim}(Y, \mathbb{Z})(-1)$ where the first orthogonal complement is taken in the Mukai lattice of $\mathcal{A}_Y$. As $v\in\widetilde{H}(\mathcal{A}_Y, \bbZ)$ lies in the algebraic part we deduce from both the isometries that $T(M_v(\sigma))\simeq T(Y)(-1)$. When \(v\) is of O'Grady type, then $\widetilde{M}_v(\sigma)$ is given by a blow-up on the singular locus of \(M_v(\sigma)\). By Voisin's blow up formula \cite[Theorem 7.31]{VoisinHodgeTheoryI} we know that the center of the blow-up contributes only to algebraic classes, so that $T(\widetilde{M}_v(\sigma))\simeq T(M_v(\sigma))\simeq T(Y)(-1)$.
\end{proof}
The following lemma tells us that when \(Y\) is very general, a primitive Mukai vector \(v\in A(\mathcal{A}_Y)\) is determined by its square up to isometry.
\begin{lemma}\label{lemma:diofanteo}
    Let $v, w\in A_2$ be two primitive elements with the same square. Then $v^{\perp}\simeq w^{\perp}$ as lattices.
\end{lemma}
\begin{proof}
    Let $v=ae_1+be_2$ for a basis $\{e_1, e_2\}$ of $A_2$. Then a generator of $v^{\perp}$ will be a primitive element $u=xe_1+ye_2$ such that $(2a-b)x+(2b-a)y=0$. We call $p:=(2a-b)$ and $q:=(2b-a)$. The diophantine equation $px+qy=0$ has solutions \begin{equation*}
        x=k\frac{q}{g},\ y=-k\frac{p}{g}, \qquad g:=\gcd(p,q),\quad k\in\bbZ.
    \end{equation*}
    The only primitive solutions are thus given by $k=\pm 1$. In both cases we get \begin{equation*}
        u^2=\frac{2}{g^2}\left(q^2+p^2+pq\right)=\frac{6}{g^2}\left(a^2+b^2-ab\right)=\frac{3}{g^2}v^2.
    \end{equation*}
    Note that $p\equiv q\equiv 2(a+b) \mod 3$. Moreover, $g\mid (2p+q)=3a$ and $g\mid (2q+p)=3b$. As $v$ is primitive we know that $\gcd(a,b)=1$, therefore $g\mid 3$, i.e. $g\in\left\{\pm 1,\pm 3\right\}$. Now, we have the identity $a^2+b^2-ab\equiv(a+b)^2 \mod 3$. Therefore,\begin{equation*}
        0\equiv v^2=2(a^2+b^2-ab) \mod 3\iff 0\equiv a+b\equiv p \equiv q\mod 3\iff g=\pm 3.
    \end{equation*}
    This shows that the square of the primitive element $u$ that generates the orthogonal complement of a primitive vector $v\in A_2$ is $u^2=\frac{v^2}{3}$ if $v^2\equiv 0\mod 3$ or, else, $u^2=3v^2$. In both cases it is uniquely determined by the square of $v$. 
\end{proof}
\begin{remark}
    If $v=kv_0$ for $k\in\bbZ$ and $v_0$ primitive we get that an element $u\in v^{\perp}\iff u\in v_0^{\perp}$, so the same result of \autoref{lemma:diofanteo} holds for $v=kv_0$, $w=kw_0$ with $v_0, w_0$ primitive vectors with the same square.
\end{remark}

Now we are ready to state and prove the following proposition in the same fashion of \autoref{prop:LSV}.

\begin{proposition}\label{prop:LEquivalenceKuznetsov}
    Let $Y$ and $Y'$ be two cubic fourfolds and assume that $Y$ is very general. Let $v\in A(\mathcal{A}_Y)$ and $w\in A(\mathcal{A}_{Y'})$ be either primitive with same square or of O'Grady type. Let $\sigma_Y\in\Stab^{\dagger}(Y)$ and $\sigma_{Y'}\in\Stab^{\dagger}(Y')$ be a $v$-generic and a $w$-generic stability condition, respectively. Suppose that \(M_v(\sigma_Y)\) and \(M_w(\sigma_{Y'})\) are L-equivalent, then $Y\cong Y'$ and \(M_v(\sigma_Y)\) is birational to \(M_w(\sigma_{Y'})\). In case \(M_v(\sigma_Y)\) and \(M_w(\sigma_{Y'})\) are of K3\(^{[n]}\) type, then they are also D-equivalent.
\end{proposition}
\begin{proof}
    By \autoref{trascendente_Kuz} we get that $T(M_v(\sigma_Y))$ is Hodge anti-isometric to $T(Y)$ and $T(M_w(\sigma_{Y}))$ to $T(Y')$. Since $Y$ is very general we know by \autoref{rmk:simple_end_algebra} that \(\End(T(X))=\bbZ\). Since $Y$ is very general, we have \(T(Y)=H_{prim}^4(Y,\bbZ)\) and similarly for \(Y'\). From the fact that \(T(Y)\) and \(T(Y')\) are isometric of signature \((22,2)\), by \autoref{lem:anti_isom} and \autoref{rmk:signatures} we get Hodge isometries $T(Y)\simeq T(M_v(\sigma_Y))(-1)\simeq T(M_w(\sigma_{Y'}))(-1)\simeq T(Y')$. From the first and the last Hodge isometries we get a Hodge isometry \(H^4(Y,\bbZ)\cong H^4(Y',\bbZ)\). The Torelli theorem for cubic fourfolds implies then then $Y\simeq Y'$. Moreover, we have $A(\mathcal{A}_Y)\simeq A(\mathcal{A}_{Y'})\simeq A_2$ and by \cite[Proposition 2.8]{AT} we get Hodge isometries $A_2^{\perp}=T(\mathcal{A}_Y)\cong  H^4_{prim}(Y, \mathbb{Z})(-1)\simeq H^4_{prim}(Y', \mathbb{Z})(-1)\cong T(\mathcal{A}_{Y'})$.

    As $v, w\in A_2$ share the same square and are either primitive or of O'Grady type we deduce from \autoref{lemma:diofanteo}, with the fact that \(O(A_2)\to O(D({A_2}))\) is surjective, that they have isomorphic embeddings in $A_2$. So, we can deduce the following chain of Hodge isometries 
    \begin{equation}\label{eq:isoperpMukai}
    H^2(M_v(\sigma_Y), \mathbb{Z})\simeq v^{\perp}\simeq w^{\perp}\simeq H^2(M_w(\sigma_{Y'}), \mathbb{Z}).
    \end{equation}
    
    If $v$ and $w$ are of O'Grady type then the Hodge isometry of (\ref{eq:isoperpMukai}) induces a Hodge isometry $H^2(\widetilde{M}_v(\sigma_Y))\simeq H^2(\widetilde{M}_w(\sigma_{Y'}))$. Moreover, by \cite{onorati_monodromy} the monodromy group of a manifold of OG10 type \(X\) is given by \(O^+(H^2(X,\bbZ))\), thus, up to composing by \(-\id\) we can assume that the Hodge isometry is monodromy. The Hodge theoretic Torelli theorem for hyper-Kähler manifolds \cite[Theorem 1.3]{markman2011survey} implies that  \(M_v(\sigma_Y)\) is birational to \(M_w(\sigma_{Y'})\). 
    
    If $v$ and $w$ are primitive, then \(M_v(\sigma_Y)\) and \(M_w(\sigma_{Y'})\) are of type $K3^{[n]}$ with \(n=\frac{v^2+2}{2}\). By the previous discussion, the Hodge isometry \(T(\mathcal{A}_Y)\cong T(\mathcal{A}_{Y'})\) extends to a Hodge isometry of the Mukai lattices \(\widetilde{H}(\mathcal{A}_Y,\bbZ)\cong \widetilde{H}(\mathcal{A}_{Y'},\bbZ)\) restricting to (\ref{eq:isoperpMukai}) with the opportune identifications, and hence \cite[Corollary 9.9]{markman2011survey} yields that \(M_v(\sigma_Y)\) is birational to \(M_w(\sigma_{Y'})\). By \cite[Theorem 1.2]{maulik2025d} they are also D-equivalent.
\end{proof}
\begin{remark}
    By \cite[Proposition 6.7]{LPZ_OG10} the OG10 manifold $\widetilde{M}_v(\sigma)$ for $v$ of O'Grady type provides a compactification of the Lagrangian fibration of twisted intermediate Jacobians. Thus, $\widetilde{M}_v(\sigma)$ is birational to the twisted LSV manifold $J^t(Y)$ \cite{voisin2016hyper}. In particular, when \(Y\) is very general, for this choice of $v$, we have that \autoref{prop:LEquivalenceKuznetsov} and \autoref{prop:LSV} are equivalent. 
\end{remark}
The tight relation between the Hodge structures of $M_v(\sigma)$ and the underlying cubic fourfold $Y$ yields also the following.
\begin{corollary}\label{cor:Kuz_special_cubics}
     Suppose that in the hypothesis of \autoref{prop:LEquivalenceKuznetsov} we ask $Y$ to be very general in $\mathcal{C}_d$ for $d$ not divisible by 9. Then \(Y\) and \(Y'\) are FM partners.
\end{corollary}
\begin{proof}
    The proof of the fact that \(Y\) and \(Y'\) are FM partners is the same of \autoref{thm:FM_equiv}, once we note that by \autoref{lem:anti_isom} we have that $T(Y)(-1)\simeq T(M_v(\sigma))\simeq T(\mathcal{A}_Y)$.
\end{proof}

\subsection{Hyper-Kähler manifolds from K3 surfaces}

\subsubsection{Desingularizations of moduli spaces of sheaves on a K3}

We firstly consider the case where the hyper-Kähler manifolds are birational to desingularizations of moduli spaces of sheaves on K3 surfaces.

\begin{proposition}\label{prop:OG10_moduli_spaces}
Let \(X\) be a manifold of OG10 type with \(\NS(X)=U\oplus A_2(-1)\), \(\disc T(X)=1\) and \(\End(T(X))=\bbZ\). Let \(X'\) be another manifold of OG10 type which is L-equivalent to \(X\), then \(X\) and \(X'\) are birational. Moreover, \(X\) and \(X'\) are D-equivalent.
\end{proposition}

\begin{proof}
Since \(\disc T(X)=1\), the primitive vector \(\sigma\in A_2(-1) \subset \NS(X)\) with \(\sigma^2=-6\) and of divisibility \(3\) in \(A_2(-1)\) has then divisibility \(3\) also in \(H^2(X,\bbZ)\).
    By \cite[Theorem 3.13]{FGG:OG10_as_moduli} we have that \(X\) is birational to a moduli space of sheaves on a K3 surface. Moreover, we have \(\rk T(X)=20>13\) and by \autoref{lem:big_transcendental} there is a Hodge isometry \(T(X)\cong T(X')\) in view of \autoref{rmk:lattices_for_defos}.
    In particular, \(T(X')\) is also unimodular and from this it follows that \(\NS(X')\) has the same discriminant form as \(\NS(X)\) (which is the one of \(H^2(X,\bbZ)\)), hence \(\NS(X)\) and \(\NS(X')\) are in the same genus. 
    Since the two lattices are even indefinite and \(\rk \NS(X)\geq l(\NS(X))+2\), we can extend the Hodge isometry \(T(X)\cong T(X')\) to a Hodge isometry \(H^2(X,\bbZ)\cong H^2(X',\bbZ)\) by \autoref{lem:extension_to_Hodge_isom2}. Up to composing it with \(-\id\), the Hodge isometry can be taken to be monodromy by \cite[Theorem 5.4]{onorati_monodromy} and the Torelli Theorem for hyper-Kähler manifolds \cite[Theorem 1.3]{markman2011survey} implies that \(X\) and \(X'\) are birational. 
    
    In particular, they are birational to a moduli space of sheaves on a K3 surface and then by \cite[Theorem 4.5.1]{halpern2020derived} they are D-equivalent.

\end{proof}

\subsubsection{Hilbert schemes of points on a K3 surface}

We now consider manifolds of K3\(^{[n]}\) type that are birational so Hilbert schemes of points on a K3 surface.

\begin{proposition}\label{prop:hillbert_scheme}
    Let \(X\) birational to \(S^{[n]}\) for a K3 surface \(S\) with \(U\subseteq \NS(S)\), and assume that \(\End(T(X))=\bbZ\) and \(\rk T(X)\not= 4\). If \(X'\) is a manifold of K3\(^{[n]}\) type which is L-equivalent to \(X\) and \(\disc T(X)=\disc T(X')=m\) is coprime to \(2(n-1)\), then \(X\) and \(X'\) are D-equivalent.
\end{proposition}

\begin{proof}
Since \(\disc T(X)=\disc T(X')\), \(\End(T(X))=\bbZ\) and \(\rk T(X)\not=4\), then by \autoref{lem:anti_isom} and \autoref{rmk:signatures} we have a Hodge isometry \(T(X)\cong T(X')\).

    Firstly, we want to show that \(D({\NS(X)})\) is isometric to \(D({\NS(X')})\).
    By hypothesis, the discriminant group \(D({T(X')})\) does not contain any subgroup isomorphic to \(\mathbb{Z}/2(n-1)\mathbb{Z}\) so that the primitive embedding \(T(X')\subset H^2(X',\bbZ)\) has trivial embedding subgroup and then the discriminant of the orthogonal complement \( \NS(X')\) is given by \(D({\NS(X')})\cong D({T(X')})(-1)\oplus D({[-2(n-1)]})\), which is isometric to the discriminant form of \(\NS(X)\cong \NS(S)\oplus [-2(n-1)]\) from the fact that \(D(\NS(S))\cong D(T(S))(-1)\cong D(T(X))(-1)\).
    
    From the fact that \(U \subseteq \NS(X)\) we have \(\rk \NS(X)\geq l(\NS(X))+2\). 
Moreover, since \(\NS(X)\) and \(\NS(X')\) have the same genus we are in the place to apply \autoref{lem:extension_to_Hodge_isom2} to get a Hodge isometry \(H^2(X,\bbZ)\cong H^2(X',\bbZ)\). Up to composing with \(-\id\) the Hodge isometry can be chosen to be orientation-preserving. Moreover, since \(\End(T(X))=\bbZ\) and \(O(\NS(X))\to O(D(\NS(X)))\) is surjective with the fact that \(D({\NS(X)})\cong D({T(X)})(-1)\oplus D({[-2(n-1)]})\), the Hodge isometry can be chosen to act as \(\pm 1\) on the discriminant and hence by \cite[Corollary 9.9]{markman2011survey} we have that \(X\) and \(X'\) are birational. Now, by \cite[Theorem 1.2]{maulik2025d} we know that they are also D-equivalent.

\end{proof}


\color{black}

 \bibliographystyle{abbrv}
\bibliography{References}

\end{document}